\documentclass[11pt,reqno]{amsart}
\usepackage{graphicx}
\usepackage{amssymb}
\usepackage{latexsym}

\newtheorem{theorem}{Theorem}[section]
\newtheorem{corollary}[theorem]{Corollary}
\newtheorem{lemma}[theorem]{Lemma}

\newtheorem{conjecture}[theorem]{Conjecture}
\newtheorem{question}[theorem]{Question}

\newtheorem{historical background}[theorem]{Historical Background}

\newtheorem{example}[theorem]{Example}

\newtheorem{definition}[theorem]{Definition}

\DeclareMathOperator\st{\operatorname{st}}

\def\sD {{\mathcal D}}

\def\sF {{\mathcal F}}

\def\sO {{\mathcal O}}

\def\sU {{\mathcal U}}
\def\sV {{\mathcal V}}

\def\min {\mathrm{min}}
\def\sup {\mathrm{sup}}

\def\< {{\langle}}
\def\> {{\rangle}}

\begin{document}
%%%%%%%%%%%%%%%%%%%%%%%%%%%%%%%%%%%%%%%%%%%%%%%%%%%%%%%%%%%%
%%%%%%%%%%%%%%%%%%%%%%%%%%%%%%%%%%%%%%%%%%%%%%%%%%%%%%%%%%%%
%\noindent                                             %%%%%%
%\begin{picture}(150,36)                               %%%%%%
%\put(5,20){\tiny{To be submitted to}}                 %%%%%%
%\put(5,7){\textbf{Topology and its applications}}     %%%%%%
%\put(0,0){\framebox(172,34){}}                        %%%%%%
%\put(2,2){\framebox(168,30){}}                        %%%%%%
%\end{picture}                                         %%%%%%
%%%%%%%%%%%%%%%%%%%%%%%%%%%%%%%%%%%%%%%%%%%%%%%%%%%%%%%%%%%%
%%%%%%%%%%%%%%%%%%%%%%%%%%%%%%%%%%%%%%%%%%%%%%%%%%%%%%%%%%%%

\vspace{0.5in}

\title[Cardinalities of spaces with regular $G_\kappa$-diagonals]
{Cardinalities of weakly Lindel\"of spaces with regular $G_\kappa$-diagonals}

%    Information for first author:
\author{Ivan S. Gotchev}
\address{Department of Mathematical Sciences, Central Connecticut Sta\-te 
University, 1615 Stanley Street, New Britain, CT 06050}
%    Current address (if needed):
%\curraddr{}
\email{gotchevi@ccsu.edu}
\thanks{The author is grateful to the Mathematics Department at the Universidad Aut\'onoma Metropolitana, Mexico City, Mexico, for their hospitality and support during his sabbatical visit of UAM in the spring semester of 2015.}

%    Information for second author (if needed):
%\author{}
%\address{}
%    Current address (if needed):
%\curraddr{}
%\email{}
%\thanks{The second author was supported in part by NSF Grant \#000000.}

%    Information for third author (if needed):
%\author{}
%\address{}
%\email{}
%\thanks{Support information for the third author.}

%    General info


\begin{abstract}
For a Urysohn space $X$ we define the \emph{regular diagonal degree} $\overline{\Delta}(X)$ of $X$ 
to be the minimal infinite cardinal $\kappa$ such that $X$ has a regular $G_\kappa$-diagonal i.e. there 
is a family $(U_\eta:\eta<\kappa)$ of open neighborhoods of $\Delta_X=\{(x,x)\in X^2:x\in X\}$ in 
$X^2$ such that $\Delta_X = \bigcap_{\eta<\kappa} \overline{U}_\eta$.

In this paper we show that if $X$ is a Urysohn space then: 
(1) $|X|\leq  2^{c(X)\cdot\overline{\Delta}(X)}$; 
(2) $|X|\leq  2^{\overline{\Delta}(X)\cdot 2^{wL(X)}}$; 
(3) $|X|\le wL(X)^{\overline{\Delta}(X)\cdot\chi(X)}$; and 
(4) $|X|\le aL(X)^{\overline{\Delta}(X)}$; 
where $\chi(X)$, $c(X)$, $wL(X)$ and $aL(X)$ are respectively the character, the cellularity, the weak 
Lindel\"of number and the almost Lindel\"of number of $X$. 

The first inequality extends to the uncountable case Buzyakova's result that the cardinality of a 
ccc-space with a regular $G_\delta$-diagonal does not exceed $2^\omega$.
It follows from (2) that every weakly Lindel\"of space with a regular $G_\delta$-diagonal has cardinality 
at most $2^{2^\omega}$.  

Inequality (3) implies that when $X$ is a space with a regular $G_\delta$-diagonal then 
$|X|\le wL(X)^{\chi(X)}$. This improves significantly Bell, Ginsburg and Woods inequality 
$|X|\le 2^{\chi(X)wL(X)}$ for the class of normal spaces with regular $G_\delta$-diagonals. 
In particular (3) shows that the cardinality of every first countable space with a regular 
$G_\delta$-diagonal does not exceed $wL(X)^\omega$.

For the class of spaces with regular $G_\delta$-diagonals (4) improves Bella and Cammaroto inequality 
$|X|\le 2^{\chi(X)\cdot aL(X)}$, which is valid for all Urysohn spaces. Also, it follows from (4) that the 
cardinality of every space with a regular $G_\delta$-diagonal does not exceed $aL(X)^\omega$. 
\end{abstract}

\maketitle

\section{Introduction}\label{s1}
Perhaps the two most famous results involving cardinal functions are Arhangel'skii's and Hajnal-Juh\'as' 
theorems asserting that if $X$ is a Hausdorff space then $|X|\le 2^{\chi(X)L(X)}$ \cite{Arh69} and 
$|X|\le 2^{\chi(X)c(X)}$ \cite{Juhasz80}, where $\chi(X)$, $L(X)$ and $c(X)$ denote respectively the 
character, Lindel\"of number and cellularity of $X$. 

Bell, Ginsburg and Woods showed in \cite{BGW78} that if $X$ is a normal space then 
\begin{equation}\label{Eq1}
|X|\le 2^{\chi(X)wL(X)}
\end{equation}
where $wL(X)$ is the weak Lindel\"of number of $X$. Since $wL(X)\le L(X)$ and $wL(X)\le c(X)$, 
(\ref{Eq1}) generalizes (for the class of normal spaces) Arhangel'skii's and Hajnal-Juh\'as' inequalities. 
In the same paper the authors constructed an example (see \cite[Example 2.3]{BGW78}) showing that 
for Hausdorff spaces the gap between $|X|$ and $2^{\chi(X)wL(X)}$ could be arbitrarily large and they 
asked (see \cite[4.1]{BGW78}) if (\ref{Eq1}) holds true for all regular $T_1$-spaces. To the best of our 
knowledge this question is still open (see \cite[Question 1]{Hodel06}). In this paper we give a partial 
answer to their question by showing that for every space $X$ with a regular $G_\delta$-diagonal even 
the stronger inequality $|X|\le wL(X)^{\chi(X)}$ is true.

In 1977, Ginsburg and Woods proved (see \cite{GW77}) that if $X$ is a $T_1$-space then 
$|X|\le 2^{e(X)\Delta(X)}$, where $e(X)$ and $\Delta(X)$ denote respectively the extent and the 
diagonal degree of $X$. As a corollary of that inequality the authors obtained that if $X$ is a 
collectionwise Hausdorff space then 
\begin{equation}\label{Eq2}
|X|\le 2^{c(X)\Delta(X)}.
\end{equation}
They also noticed (see \cite[Example 2.4]{GW77}) that the Kat\v{e}tov extension $k\omega$ of the 
countable discrete space $\omega$ is an example of a Urysohn space (every two points have disjoint 
closed neighborhoods) for which $|k\omega|>2^{c(k\omega)\Delta(k\omega)}$ and they asked if 
(\ref{Eq2}) was true for every regular $T_1$-space \cite[Question 2.5]{GW77}. In 1978 Arhangel'skii 
independently asked the countable version of that same question: Is it true that if $X$ is a regular 
ccc-space with a $G_\delta$-diagonal then $|X|\le 2^\omega$ (see \cite[p. 91, Question 16]{Arh78}). 
Shakhmatov answered their question in \cite{Sha84} by showing that there is no upper bound for the 
cardinality of completely regular ccc-spaces with $G_\delta$-diagonals. Then Arhangel'skii asked what if 
``$G_\delta$-diagonal" is replaced by ``regular $G_\delta$-diagonal" \cite{{Buz06}}. Buzyakova 
answered that question by proving the following theorem:

\begin{theorem}[\cite{Buz06}]\label{B}
The cardinality of a ccc-space with a regular $G_\delta$-diagonal does not exceed $2^\omega$. 
\end{theorem}

In this paper we show that if $X$ is a Urysohn space then: 

\begin{equation}\label{EqB}
|X|\leq  2^{c(X)\cdot\overline{\Delta}(X)};
\end{equation}
\begin{equation}\label{Eq3}
|X|\leq  2^{\overline{\Delta}(X)\cdot 2^{wL(X)}};
\end{equation}
\begin{equation}\label{Eq4}
|X|\le wL(X)^{\overline{\Delta}(X)\cdot{\chi(X)}}; {\text{and}}
\end{equation}
\begin{equation}\label{Eq5}
|X|\le aL(X)^{\overline{\Delta}(X)}.
\end{equation}

Inequality (\ref{EqB}) extends to the uncountable case Theorem \ref{B}. It follows from (\ref{Eq3}) that 
$2^{2^\omega}$ is an upper bound for the cardinality of weakly Lindel\"of spaces with regular 
$G_\delta$-diagonals. This shows again that ``regular $G_\delta$-diagonal'' is a much more restrictive 
property than ``$G_\delta$-diagonal''.

It follows from (\ref{Eq4}) that when $X$ is a space with a regular $G_\delta$-diagonal then 
$|X|\le wL(X)^{\chi(X)}$. This improves significantly Bell, Ginsburg and Woods inequality for the class of 
normal spaces with regular $G_\delta$-diagonals. In particular (\ref{Eq4}) shows that the cardinality of 
every first countable space with a regular $G_\delta$-diagonal does not exceed $wL(X)^\omega$.

For the class of spaces with regular $G_\delta$-diagonals (\ref{Eq5}) improves Bella and Cammaroto 
inequality $|X|\le 2^{\chi(X)\cdot aL(X)}$, which is valid for all Urysohn spaces \cite{BelCam88}. Also, it 
follows from (\ref{Eq5}) that the cardinality of every space with a regular $G_\delta$-diagonal does not 
exceed $aL(X)^\omega$. 

\section{Definitions}\label{s2}

Throughout this paper $\omega$ is (the cardinality of) the set of all non-negative integers, $\xi$ and 
$\eta$ are ordinals and $\tau$, $\mu$ and $\kappa$ are infinite cardinals. The cardinality of the set $X$ 
is denoted by $|X|$ and $\Delta_X = \{(x,x)\in X^2:x\in X\}$ is the \emph{diagonal} of $X$. If $\sU$ is 
a family of subsets of $X$, $x\in X$, and 
$G\subset X$ then  $\st(G,\sU) = \bigcup\{U\in \sU:U\cap G\neq \emptyset\}$. 
When $G=\{x\}$ we write $\st(x,\sU)$ instead of $\st(\{x\},\sU)$. If $n\in\omega$, 
$\st^n(G,\sU)=\st(\st^{n-1}(G,\sU),\sU)$ and $\st^0(G,\sU)=G$.

All spaces are assumed to be topological $T_1$-spaces. For a subset $U$ of a space $X$ the closure of 
$U$ (in $X$) is denoted by $\overline{U}$. $F\subset X$ is called \emph{regular-closed} (in $X$) if there 
is open $U\subset X$ such that $F=\overline{U}$. As usual, $\chi(X)$ and $\psi(X)$ denote respectively 
the character and the pseudocharacter of $X$. The \emph{closed pseudo-character} $\psi_c(X)$ 
(defined only for Hausdorff spaces $X$)  is the smallest infinite cardinal $\kappa$ such that for each 
$x\in X$, there is a collection $\{V(\eta,x):\eta<\kappa\}$ of open neighborhoods of $x$ such that 
$\bigcap_{\eta<\kappa}\overline{V}(\eta,x) = \{x\}$ \cite{Schr93}. 
The \emph{Hausdorff pseudo-character} of X, denoted $H\psi(X)$, is the smallest infinite cardinal 
$\kappa$ such that for each $x \in X$, there is a collection $\{V(\eta,x): \eta <\kappa\}$ of open 
neighborhoods of $x$ such that if $x \neq y$, then there exists $\eta,\xi < \kappa$ such that 
$V(\eta,x)\cap V(\xi,y) = \emptyset$ \cite{Hodel06}. 

The \emph{Lindel\"off number} of $X$ is $L(X)=\min\{\kappa:$ every open cover of $X$ has a 
subcover of cardinality $\leq\kappa\}+\omega$. The \emph{weak Lindel\"off number} of $X$, denoted 
$wL(X)$, is the smallest infinite cardinal $\kappa$ such that every open cover of $X$ has a subcollection 
of cardinality $\le\kappa$ whose union is dense in $X$. If $wL(X)=\omega$ then $X$ is called 
\emph{weakly Lindel\"of}. The \emph{almost Lindel\"off number} of $X$, denoted $aL(X)$, is the 
smallest infinite cardinal $\kappa$ such that for every open cover $\sU$ of $X$ there is a subcollection 
$\sU_0$ such that $|\sU_0|\le\kappa$ and $\bigcup\{\overline{U}:U\in\sU_0\}=X$. If $aL(X)=\omega$ 
then $X$ is called \emph{almost Lindel\"of}. 
$e(X)=\sup\{|D|:D\subseteq X$ is closed and discrete$\}+\omega$ is the \emph{extent} of $X$. 
A pairwise disjoint collection of non-empty open sets in $X$ is called a \emph{cellular family}. The 
\emph{cellularity} of $X$ is $c(X)=\sup\{|\sU|:\sU$ a cellular family in $X\}+\omega$. If 
$c(X)=\omega$ then it is called that $X$ satisfies the \emph{countable chain condition} (or has the 
\emph{ccc}) property. 

A space $X$ has a \emph{$G_\kappa$-diagonal} if there is a family $(U_\eta:\eta<\kappa)$ of open sets 
in $X^2$ such that $\Delta_X=\bigcap_{\eta<\kappa} U_\eta$; if  
$\Delta_X = \bigcap_{\eta<\kappa} U_\eta =  \bigcap_{\eta<\kappa} \overline{U}_\eta$ then 
$X$ has a \emph{regular $G_\kappa$-diagonal}. When $\kappa=\omega$ then 
$X$ has a $G_\delta$-diagonal (respectively, regular $G_\delta$-diagonal).  The \emph{diagonal 
degree} of $X$, denoted $\Delta(X)$, is the smallest infinite cardinal $\kappa$ such that $X$ has a 
$G_\kappa$-diagonal (hence $\Delta(X)=\omega$ if and only if $X$ has a $G_\delta$-diagonal).

The following observation is well-known and easy to prove (see e.g. \cite[Lemma 4.6]{ComGot09}).

\begin{lemma}
$X$ has a diagonal which is the intersection of some of its regular-closed neighborhoods if and only if 
$X$ is a Urysohn space.
\end{lemma}

\begin{definition}
For a Urysohn space $X$ we define the \emph{regular diagonal degree} $\overline{\Delta}(X)$ of $X$ to 
be the minimal infinite cardinal $\kappa$ such that $X$ has a regular $G_\kappa$-diagonal.
\end{definition}

Let $n$ be a positive integer. $X$ has a \emph{rank $n$-diagonal} (a \emph{strong rank $n$-diagonal}) 
if there is a sequence $\{\sU_m : m <\omega\}$ of open covers of $X$ such that for all $x\ne y$, there 
is some $m<\omega$ such that $y\notin \st^n(x, \sU_m)$ ($y\notin \overline{\st^n(x,\sU_m)}$) 
(\cite{ArhBuz06}, \cite{BBR14}). Spaces $X$ with rank $n$-diagonals and strong rank $n$-diagonals 
were introduced and first studied in \cite{Ishii70} under the names ``spaces with 
$G_\delta(n)$-diagonals" and ``spaces with $\overline{G}_\delta(n)$-diagonals". Clearly the spaces with 
strong rank 1-diagonals or, equivalently, the spaces with  $\overline{G}_\delta(1)$-diagonals, are exactly 
the spaces with $G^*_\delta$-diagonals introduced and studied in \cite{Hodel71}.

The \emph{rank} (\emph{strong rank}) of the diagonal of a space $X$ is defined as the greatest natural 
number $n$ such that $X$ has a rank $n$-diagonal (strong rank $n$-diagonal), if such a number $n$ 
exists. The rank (strong rank) of the diagonal of $X$ is infinite, if $X$ has a rank $n$-diagonal (strong 
rank $n$-diagonal) for every $n\ge 1$ (\cite{ArhBuz06}, \cite{BBR14}).

Condensations are one-to-one and onto continuous mappings. A space $X$ is \emph{submetrizable} if it 
condenses onto a metrizable space \cite{ArhBuz06}, or equivalently, $(X,\tau)$ is \emph{submetrizable} 
if there exists a topology $\tau'$ on $X$ such that $\tau'\subset \tau$ and $(X,\tau')$ is metrizable 
\cite{Gru84}.

A sequence $(\sU_n:n<\omega)$ of open covers of a space $X$ is a \emph{development} for $X$ if for 
each $x\in X$, the set $\{\st(x,\sU_n):n<\omega\}$ is a base at $x$. A \emph{developable space} is a 
space that has a development. A \emph{Moore space} is a regular developable space \cite{Gru84}.

For definitions not given here we refer the reader to \cite{E89}, \cite{Juhasz80} or \cite{Hodel84}.

\section{Preliminary results}\label{s3}

The following three observations are well-known. 

\begin{lemma}
For every Hausdorff space $X$ 
$$\psi(X)\leq \psi_c(X)\leq H\psi(X)\leq \chi(X).$$
\end{lemma}

\begin{lemma}
$wL(X)\le aL(X) \le L(X)$.
\end{lemma}

\begin{lemma}[\cite{Hodel84}]
$wL(X)\leq c(X)$.
\end{lemma}

In 1961 Ceder made the following observation.

\begin{lemma}[\cite{Ceder61}]\label{Ceder}
A space $X$ has a $G_\delta$-diagonal if and only if there is a sequence $(\sU_n:n<\omega)$ of open 
covers of $X$ such that if $x\in X$, then $\{x\} = \bigcap_{n<\omega} \st(x,\sU_n)$. 
\end{lemma}

It is then clear that if a space $X$ has a $G_\delta$-diagonal then $\psi(X)=\omega$ and that $X$ has 
a $G_\delta$-diagonal if and only if $X$ has a rank 1-diagonal.

The case $\kappa=\omega$ of the following lemma was proved by Zenor in \cite{Zenor72}. Below we 
show that Zenor's proof works also for every $\kappa>\omega$.

\begin{lemma}\label{ZG} Let $\kappa$ be an infinite cardinal. 
A space $X$ has a regular $G_\kappa$-diagonal if and only if there is a family $(\sU_\eta:\eta<\kappa)$ 
of open covers of $X$ such that if $x$ and $y$ are distinct points of $X$, then there is $\eta<\kappa$ 
and open sets $U\in \sU_\eta$ and $V\in \sU_\eta$ containing $x$ and $y$ respectively, such that no 
member of $\sU_\eta$ intersects both $U$ and $V$. 
\end{lemma}

\begin{proof} 
Suppose that $X$ has a regular $G_\kappa$-diagonal and let $(W_\eta:\eta<\kappa)$ be
a family of open sets in $X^2$ such that 
$\Delta_X= \bigcap_{\eta<\kappa}W_\eta = \bigcap_{\eta<\kappa}\overline{W}_\eta$. For each 
$\eta<\kappa$, let $\sU_\eta = \{U: U$ is an open subset of $X$ such that 
$U\times U\subset W_\eta\}$. To see that the family $(\sU_\eta:\eta<\kappa)$ is as required
let $x$ and $y$ be a pair of distinct points of $X$. Then there exists $\eta<\kappa$ such 
that $(x, y)$ is not in $\overline{W}_\eta$ and open sets $U$ and $V$ in $X$, which contain $x$ 
and $y$ respectively, such that $U\times V$ does not intersect $W_\eta$, 
$U\times U \subset W_\eta$, and $V\times V\subset W_\eta$. To see that no member of $\sU_\eta$ 
intersects both $U$ and $V$, suppose otherwise; that is, suppose that $W$ is a member of $\sU_\eta$, 
$p$ is a point of $W\cap U$ and $q$ is a point of $W\cap V$. Then $(p,q)$ is a point of 
$W_\eta \cap(U \times V)$, which is a contradiction.

Now, suppose that $\sU_\eta$ is a family of open covers of $X$ as described in the lemma. For each 
$\eta<\kappa$, let $W_\eta = \bigcup_{\eta<\kappa}\{U \times U: U\in \sU_\eta\}$. Clearly, 
$\Delta_X\subset \bigcap_{\eta<\kappa}W_\eta$. To see that 
$\Delta_X= \bigcap_{\eta<\kappa}\overline{W}_\eta$, let $x$ and $y$ be distinct points of $X$. Then 
there is $\eta<\kappa$ and open sets $U,V\in \sU_\eta$ containing $x$ and $y$ respectively, such that 
no member of $\sU_\eta$ intersects both $U$ and $V$. It must be the case that $W_\eta$ does not 
intersect $U\times V$.
\end{proof}

In fact the proof of Lemma \ref{ZG} shows a little bit more. 

\begin{corollary}\label{C1}
If a space $X$ has a regular $G_\kappa$-diagonal, for some infinite cardinal $\kappa$, then there is a 
family $(\sU_\eta:\eta<\kappa)$ of open covers of $X$ such that

{\rm (a)} if $x$ and $y$ are distinct points of $X$, then there exist $\eta<\kappa$ and open sets 
$U_\eta(x,y),U_\eta(y,x)\in \sU_\eta$ containing $x$ and $y$ respectively, such that 
$U_\eta(y,x)\cap\overline{\st(U_\eta (x,y),\sU_\eta)}=\emptyset$;

{\rm (b)} if $x\in X$ then 
$\{x\} = \bigcap_{\eta<\kappa} \st(x,\sU_\eta)= \bigcap_{\eta<\kappa} \overline{\st(x,\sU_\eta)}$. 
\end{corollary}

\begin{proof}
Let $(\sU_\eta:\eta<\kappa)$ be a sequence of open covers of $X$ as in Lemma \ref{ZG}. 

(a) If $x,y \in X$ are distinct points then according to Lemma \ref{ZG} there is $\eta<\kappa$ and open 
sets $U_\eta(x,y), U_\eta(y,x)\in\sU_\eta$ containing $x$ and $y$ respectively, such that no member of 
$\sU_\eta$ intersects both $U_\eta(x,y)$ and $U_\eta(y,x)$. Then $y\in U_\eta(y,x)$ and 
$U_\eta (y,x)\cap \st(U_\eta (x,y),\sU_n)=\emptyset$. Therefore 
$U_\eta(y,x)\cap\overline{\st(U_\eta (x,y),\sU_\eta)}=\emptyset$.

(b) Follows immediately from (a).
\end{proof}

It follows from Corollary \ref{C1} that if a space $X$ has a regular $G_\kappa$-diagonal then 
$\psi_c(X)=\kappa$. In particular, if a space $X$ has a regular $G_\delta$-diagonal then 
$\psi_c(X)=\omega$.

Below we mention some well-known results and open questions closely related to spaces with regular 
$G_\delta$-diagonals.

\begin{lemma}[\cite{ArhBuz06}]
Every submetrizable space $X$ has a diagonal of infinite rank.
\end{lemma}

\begin{corollary}[\cite{ArhBuz06}, \cite{BBR14}]
If the rank of the diagonal of a space $X$ is at least $3$, then $X$ has a strong rank 2-diagonal.
\end{corollary}

\begin{corollary}[\cite{ArhBuz06}, \cite{BBR14}]
If a space $X$ has a strong rank 2-diagonal, then $X$ has a regular $G_\delta$-diagonal.
\end{corollary}

\begin{corollary}[\cite{Gru84}]
Every submetrizable space $X$ has a regular $G_\delta$-diagonal.
\end{corollary}

\begin{example}[{\cite[Example 2.9]{ArhBuz06}}]\label{EAB}
There exists a separable Tychonoff Moore space with a rank 3-diagonal (hence with a regular 
$G_\delta$-diagonal) that is not submetrizable. 
\end{example}

\begin{lemma}[\cite{ArhBuz06}]\label{L2}
Every Moore space $X$ has a rank $2$-diagonal.
\end{lemma}

\begin{lemma}[\cite{Gru84}]\label{L1}
Not every Moore space $X$ has a regular $G_\delta$-diagonal.
\end{lemma}

\begin{question}[A. Bella (see \cite{BBR14}, \cite{ArhBuz06})]\label{QB}
Is every regular $G_\delta$-diagonal a rank 2-diagonal?
\end{question}

As it is noted in \cite{BBR14}, there is no example yet even of a space $X$ with a regular 
$G_\delta$-diagonal that does not have a strong rank 2-diagonal.

\begin{conjecture}[\cite{ArhBuz06}]\label{CAB}
For every natural number $n$ there is a Tychonoff space $X_n$ with a
rank $n$-diagonal that is not a rank $n + 1$-diagonal.
\end{conjecture}

In 1991 Hodel established the following result.

\begin{theorem}[\cite{Hodel91}]\label{Hodel}
If $X$ is a Hausdorff space then $|X|\leq 2^{c(X)H\psi(X)}$.
\end{theorem}

A. Bella proved in \cite{Bella87} the following theorem:

\begin{theorem}[\cite{Bella87}]\label{TB}
The cardinality of a ccc-space with a rank 2-diagonal does not exceed $2^\omega$. 
\end{theorem}

Therefore if Question \ref{QB} has a positive answer then Buzyakova's theorem (Theorem \ref{B}) will 
follow from Bella's theorem (Theorem \ref{TB}).

In \cite{Buz06} Buzyakova asked the following question (in that relation see Example \ref{EAB}):

\begin{question}[\cite{Buz06}]
Is there a ccc-space with a regular $G_\delta$-diagonal that does not condense onto
a first-countable Hausdorff space?
\end{question}

If the answer of the above question is in the negative then Buzyakova's theorem will follow immediately 
from Hodel's inequality (Theorem \ref{Hodel}).

\section{Main results}\label{s4}

Theorem \ref{GB} below extends Buzyakova's result (Theorem \ref{B}) for uncountable cardinalities. Its  
proof follows closely the original proof of Buzyakova. We begin first with a generalization for higher
cardinalities of Lemma 2.1 from \cite{Buz06}.

\begin{lemma}\label{LBG}
Let $\kappa$ be an infinite cardinal, $X$ be a space with $c(X)=\kappa$ and $U\times V$ be a 
non-empty open set in $X^2$. Let also $\mathcal{U}$ be a collection of open boxes in $X^2$ such that 
$U\times V\subset \overline{\bigcup\sU}$. Then there exists $\sV=\{U_\eta\times V_\eta:\eta<\kappa\}$ 
such that $\sV\subset \sU$, $V\subset\overline{\bigcup_{\eta<\kappa} V_\eta}$ and 
$U_\eta\times V_\eta$ meets $U\times V$ for each $\eta<\kappa$. 
\end{lemma}

\begin{proof}
Let $\sU'$ consist of all elements of $\sU$ that meet $U\times V$. Since 
$U\times V\subset \overline{\bigcup\sU}$ and $U$ is not empty, we have 
$V\subset\overline{\bigcup\{V_\xi:U_\xi\times V_\xi\in\sU'\}}$. Since $c(X)=\kappa$ and $V$ is open in 
$X$, there exists $\sV=\{U_\eta\times V_\eta:\eta<\kappa\}$ such that $\sV\subset \sU'$, 
$V\subset\overline{\bigcup_{\eta<\kappa} V_\eta}$ and $U_\eta\times V_\eta$ meets $U\times V$ for 
each $\eta<\kappa$. 
\end{proof}

\begin{theorem}\label{GB}
If $X$ is a Urysohn space then $|X|\leq  2^{c(X)\cdot\overline\Delta(X)}.$
\end{theorem}

\begin{proof}
Let $c(X)=\kappa$ and $\overline\Delta(X)=\tau$. Then $X$ is a space with a regular 
$G_\tau$-diagonal. Let $\{W_\xi:\xi<\tau\}$ be a family of open sets in $X^2$ such that 
$\Delta_X=\bigcap\{W_\xi:\xi<\tau\}=\bigcap\{\overline{W_\xi}:\xi<\tau\}$. For each $\xi<\tau$, fix a 
collection $\sU_\xi$ of open boxes in $X^2\setminus\overline{W_\xi}$ such that $|\sU_\xi|\le 2^\kappa$ 
and $X^2\setminus\overline{W_\xi}\subset\overline{\bigcup\sU_\xi}$. Such a collection exists because 
$c(X^2)\le 2^\kappa$ \cite{Kurepa62}.

Now let $\sF$ be a family of open sets in $X$ such that $F\in \sF$ if and only if 
$F=X\setminus\overline{\bigcup\{V:U\times V\in \sV\text{ for some }U\}}$, where $\sV$ is such that 
$|\sV|\le\kappa$ and $\sV\subset\sU_\xi$ for some $\xi<\kappa$. 

Since $|\sU_\xi|\le 2^\kappa$ for each $\xi<\tau$ and each $F\in \sF$ is determined by a subset $\sV$ 
of some $\sU_\xi$ with cardinality at most $\kappa$, we have $|\sF|\le \tau\cdot 2^\kappa$. To finish 
the proof that $|X|\le 2^{\kappa\cdot\tau}$ it is sufficient to show that for every $x\in X$ there exists a 
family $\{F_\xi:\xi<\tau\}\subset\sF$ such that $\{x\}=\bigcap\{F_\xi:\xi<\tau\}$.

Let $x\in X$ be fixed. For each $\xi<\tau$ we shall define $F_\xi$. Let $B$ be an open set in $X$ 
such that $x\in B$ and $B\times B\subset W_\xi$. Fix a collection $\sO$ of open boxes in 
$X^2\setminus\overline{W_\xi}$ with $|\sO|\le \kappa$ and such that:

(a) $x\in U\subset B$ for every $U\times V\in\sO$; and

(b) $\overline{\bigcup\{V:U\times V\in \sO\text{ for some }U\}}$ contains 
$\{y:(x,y)\notin \overline{W_\xi}\}$.

In order to construct such a collection we first cover $(\{x\}\times X)\setminus\overline{W_\xi}$ by 
open boxes in $X^2\setminus\overline{W_\xi}$ satisfying (a) and then using the fact that 
$c(X)\le\kappa$ we find the desired subcollection $\sO$.

For each $U\times V\in \sO$, fix $\sV_{U\times V}=\{U_\eta\times V_\eta:\eta<\kappa\}$ such that 
$\sV_{U\times V}\subset \sU_\xi$ satisfies the conclusion of Lemma \ref{LBG} and let 
$G_{U\times V}=\bigcup\{V_\eta:\eta<\kappa\}$. 

Now we shall verify the following two observations:

(1) $V_\eta\cap B=\emptyset$ for each $\eta<\kappa$. Indeed, by (a) $U\subset B$ and by the 
conclusion of Lemma \ref{LBG}, $U\cap U_\eta\ne\emptyset$. Hence $U_\eta\cap B\ne\emptyset$. 
Now suppose that $V_\eta\cap B\ne\emptyset$. Then 
$(U_\eta\times V_\eta)\cap (B\times B)\ne\emptyset$. However $B\times B\subset W_\xi$ while 
$U_\eta\times V_\eta\in\sU_\xi$ is a subset of $X^2\setminus\overline{W_\xi}$ -- contradiction.

(2) $\overline{\bigcup\{G_{U\times V}:U\times V\in\sO\}}$ contains 
$\{y:(x,y)\notin \overline{W_\xi}\}$. 
Indeed, since all $U_\eta\times V_\eta$ satisfy the conclusion of Lemma \ref{LBG}, 
$\overline{G_{U\times V}}$ contains $V$. Now the claim follows from (b).

Let $F_\xi=X\setminus\overline{\bigcup\{G_{U\times V}:U\times V\in\sO\}}$. Then $F_\xi\in \sF$. 
It follows from (1) that $x\in F_\xi$ for each $\xi<\tau$. By (2), $F_\xi$ misses 
$\{y:(x,y)\notin \overline{W_\xi}\}$. By the choice of the sets $W_\xi$, we have 
$\{x\}=\bigcap_{\xi<\tau} F_\xi$.
\end{proof}

It follows directly from our next result that there is an upper bound even for the cardinality of weakly 
Lindel\"of spaces with regular $G_\delta$-diagonals.

\begin{theorem}\label{G4}
If $X$ is a Urysohn space then $|X|\leq  2^{\overline\Delta(X)\cdot 2^{wL(X)}}.$
\end{theorem}

\begin{proof}
Let $\overline\Delta(X)\le\kappa$. Then $X$ is a space with a regular $G_\kappa$-diagonal. Therefore, 
according to Corollary \ref{C1}, there is a family $(\sU_\eta:\eta<\kappa)$ of open covers of $X$ such 
that if $x$ and $y$ are distinct points of $X$, then there exist $\eta<\kappa$ and open sets 
$U_\eta(x,y),U_\eta(y,x)\in \sU_\eta$ containing $x$ and $y$ respectively, such that 
$U_\eta(y,x)\cap\overline{\st(U_\eta (x,y),\sU_\eta)}=\emptyset$. 
Therefore for each pair of distinct points $x,y\in X$, we can fix one such $\eta$, denoted by 
$\eta(x,y)$, and open sets $U_{\eta(x,y)}(x,y)$ and $U_{\eta(x,y)}(y,x)$ with the required properties. 

Let $wL(X)\le\tau$. Then for each $\eta < \kappa$ we can find a subfamily $\mathcal{D}_\eta$ of 
$\mathcal{U}_\eta$ such that $|\mathcal{D}_\eta|\le \tau$ and 
$X = \overline{\bigcup \mathcal{D}_\eta}$. 

For $x\in X$ and $\eta<\kappa$, let $Y_\eta(x)=\{y: y\in X\setminus\{x\}, \eta(x,y)=\eta\}$. 
For $y\in Y_\eta(x)$ let $\sD_\eta^x(y)=\{U:U\in\sD_\eta, U\cap U_\eta(y,x)\ne\emptyset\}$. 
Then $U_\eta(x,y)\cap\overline{\bigcup\sD_\eta^x(y)}=\emptyset$, hence 
$x\notin \overline{\bigcup\sD_\eta^x(y)}$.
Therefore $\{x\}=\bigcap\{\bigcap\{X\setminus\overline{\bigcup\sD_\eta^x(y)}:y\in Y_\eta(x)\}:\eta<\kappa\}$.

For each $x\in X$ and $\eta<\kappa$, there are at most $2^\tau$ sets of the form $\sD_\eta^x(y)$ for 
$|\sD_\eta|\le\tau$. Hence, there are at most $2^{2^\tau}$ sets of the form 
$\bigcap\{X\setminus \overline{\bigcup\sD_\eta^x(y)}:y\in Y_\eta(x)\}$. Thus, there are at most 
$2^{\kappa\cdot 2^\tau}$ many intersections of the form $\bigcap\{\bigcap\{X\setminus\overline{\bigcup\sD_\eta^x(y)}:y\in Y_\eta(x)\}:\eta<\kappa\}$. 
Therefore we conclude that $|X|\le 2^{\kappa\cdot 2^\tau}$.
\end{proof}

\begin{corollary}\label{CGWL0}
If $X$ is a Urysohn space with a regular $G_\delta$-diagonal then $|X|\le 2^{2^{wL(X)}}$.
\end{corollary}

\begin{corollary}\label{CGWL}
If $X$ is a weakly Lindel\"of, Urysohn space with a regular $G_\delta$-diagonal then $|X|\le 2^{2^\omega}$.
\end{corollary}

With our next theorem we show how the cardinality of a Urysohn space $X$ depends on the 
cardinal functions $\chi(X)$, $wL(X)$ and $\overline{\Delta}(X)$.

\begin{theorem}\label{G3}
Let $X$ be a Urysohn space. Then $|X|\le wL(X)^{\chi(X)\cdot\overline{\Delta}(X)}$.
\end{theorem}

\begin{proof}
Take infinite cardinals $\kappa$, $\lambda$ and $\mu$ such that $\overline{\Delta}(X)\le\kappa$,  
$wL(X)\le\lambda$ and $\chi(X)\le\mu$. Then, according to Lemma \ref{ZG}, there is a family 
$\{\sU_\eta: \eta<\kappa\}$ of open covers of $X$ such that for any distinct points $x, y\in X$ we can 
find and fix an ordinal $\eta=\eta(x,y)<\kappa$ and a set $U_\eta(x,y)\in \sU_\eta$ such that 
$x\in U_\eta(x,y)$ and $y\notin\overline{\st(U_\eta(x,y),\sU_\eta)}$.  Let 
$Y_\eta(x)=\{y:y\in X\setminus\{x\},\eta(x,y)=\eta\}$. Notice that from $wL(X)\le\lambda$ it follows 
that for any $\eta < \kappa$ there exists a family $\mathcal{D}_\eta\subset \mathcal{U}_\eta$ such 
that $|\mathcal{D}_\eta|\le \lambda$ and $\overline{\bigcup \mathcal{D}_\eta}=X$. 

For every $x\in X$ fix a base $\sV_x=\{V_\xi(x):\xi<\mu\}$ of $X$ at the point $x$. Given 
$\eta<\kappa$ and $\xi<\mu$ fix $D_\xi(x)\in\sD_\eta$ such that 
$D_\xi(x)\cap V_{\xi}(x)\ne\emptyset$; observe that such an element of $\sD_\eta$ exists because the 
union of $\sD_\eta$ is dense in $X$. Let $\sD_\eta(x)=\{D_\xi(x):\xi<\mu\}$ and for each $\xi<\mu$ let 
$\sD_\eta(x,\xi)=\{D:D\in\sD_\eta(x), D\cap V_\xi(x)\ne\emptyset\}$. It is clear that for each $\xi<\mu$ 
we have  $\sD_\eta(x,\xi)\ne\emptyset$ and $|\sD_\eta(x,\xi)|\le\mu$. We claim that 
$x\in \overline{\bigcup\sD_\eta(x,\xi)}$. To see that take any open neighborhood $O$ of $x$. Then 
$O\cap V_\xi(x)$ is also an open neighborhood of $x$ and therefore there is $\xi'<\mu$ such that 
$V_{\xi'}(x)\subset O\cap V_\xi(x)$. Thus, $D_{\xi'}(x)\in\sD_\eta(x,\xi)$. Hence 
$O\cap D_{\xi'}(x)\ne\emptyset$ and therefore $O\cap (\bigcup\sD_\eta(x,\xi))\ne\emptyset$.

Notice that for every $y\in Y_\eta(x)$ there is $\xi<\mu$ such that $V_\xi(x)\subset U_\eta(x,y)$ and 
hence $y\notin \overline{\bigcup\sD_\eta(x,\xi)}$. Thus, for each $\eta<\kappa$ we have 
$x\in\bigcap\{\overline{\bigcup\sD_\eta(x,\xi)}:\xi<\mu\}$ and 
$Y_\eta(x)\cap(\bigcap\{\overline{\bigcup\sD_\eta(x,\xi)}:\xi<\mu\})=\emptyset$.
Therefore $\{x\}=\bigcap\{\bigcap\{\overline{\bigcup\sD_\eta(x,\xi)}:\xi<\mu\}:\eta<\kappa\}$.

For each $\eta<\kappa$ there are at most $\lambda^\mu$ possible subsets of $\sD_\eta$ of the type 
$\sD_\eta(x)$. For each such set, there are at most $2^\mu$ subsets of the type $\sD_\eta(x,\xi)$. 
Hence, there are at most $\lambda^\mu\cdot 2^\mu=\lambda^\mu$ possible subsets of $\sD_\eta$ of 
the type $\sD_\eta(x,\xi)$. Thus, for each $\eta<\kappa$ there are at most 
$(\lambda^\mu)^\mu=\lambda^\mu$ possible intersections 
$\bigcap\{\overline{\bigcup\sD_\eta(x,\xi)}:\xi<\mu\}$. Therefore there are at most 
$(\lambda^\mu)^\kappa$ possible intersections 
$\bigcap\{\bigcap\{\overline{\bigcup\sD_\eta(x,\xi)}:\xi<\mu\}:\eta<\kappa\}$, i.e., we proved that 
$|X|\le\lambda^{\kappa\cdot\mu}$.
\end{proof}

\begin{corollary}
If $X$ is a Urysohn space with a regular $G_\delta$-diagonal then $|X|\le wL(X)^{\chi(X)}$.
\end{corollary}

\begin{corollary}\label{CGGD}
Let $X$ be a first countable Urysohn space with a regular $G_\delta$-diagonal. Then 
$|X|\le wL(X)^{\omega}$.
\end{corollary}

\begin{corollary}
If $X$ is a weakly Lindel\"of, first countable, Urysohn space then $|X|\le 2^{\overline{\Delta}(X)}$.
\end{corollary}

\begin{corollary}
Let $X$ be a weakly Lindel\"of, developable, Urysohn space. Then $|X|\le 2^{\overline{\Delta}(X)}$.
\end{corollary}

\begin{corollary}
The cardinality of every weakly Lindel\"of, first countable, Urysohn space with a regular 
$G_\delta$-diagonal does not exceed $2^{\omega}$.
\end{corollary}

\begin{corollary}\label{C12}
Let $X$ be a weakly Lindel\"of, Moore space. If $|X|>2^\omega$ then $X$ is not submetrizable.
\end{corollary}

We note that it follows from Corollary \ref{CGGD} that if $X$ is a developable, Urysohn space with a 
regular $G_\delta$-diagonal then $|X|\le wL(X)^\omega$. Also, in the proof of 
\cite[Theorem 2]{Zenor72} among other things Zenor showed that the following lemma is true.

\begin{lemma}\label{LZ}
If $X$ is a developable Urysohn space then $X$ has a regular 
$G_\delta$-diagonal if and only if $X$ has a rank 3-diagonal.
\end{lemma}

Therefore for the class of developable Urysohn spaces Corollary \ref{CGGD} is equivalent to  
Proposition 4.7 from \cite{BBR14}: If $X$ has a rank 3-diagonal then $|X|\le wL(X)^\omega$.
In relation to that result, Theorem \ref{G3} and Corollary \ref{C12} the following questions are 
of interest: 

\begin{question}[\cite{BBR14}]\label{QB2}
Is it the case that if $X$ has a strong rank 2-diagonal then $|X|\le wL(X)^\omega$?
\end{question}

\begin{question}\label{QG1}
Is there a weakly Lindel\"of, Moore space with cardinality greater than $2^\omega$?
\end{question}

We recall that in \cite{Reed76} the author attributed to 
\cite{Zenor72} the following result: A $T_2$-space $X$ has a regular $G_\delta$-diagonal if 
and only if it has a rank 3-diagonal. (We cite that result using the current terminology. 
For the exact statement see \cite[Theorem 5]{Reed76}). But, since there is no proof of this claim in 
\cite{Zenor72}, and in relation to Questions \ref{QB} and \ref{QB2} and some of the above results 
we believe that the author of \cite{Reed76} had in mind Lemma \ref{LZ} instead.

We finish with a theorem that gives an upper bound for the cardinality of a Urysohn space $X$ as a 
function of $aL(X)$ and $\overline\Delta(X)$.

\begin{theorem}\label{G2}
If $X$ is a Urysohn space then $|X|\leq  aL(X)^{\overline\Delta(X)}.$
\end{theorem}

\begin{proof}
Let $\overline\Delta(X)=\kappa$. Then $X$ is a space with a regular $G_\kappa$-diagonal. Therefore, 
according to Corollary \ref{C1}, there is a family $(\sU_\eta:\eta<\kappa)$ of open covers of $X$ such 
that if $x$ and $y$ are distinct points of $X$, then there is $\eta<\kappa$ and an open set 
$U_\eta(x,y) \in \sU_\eta$  containing $x$ such that $y\notin \overline{\st(U_\eta (x,y),\sU_\eta)}$.

Let $aL(X)=\tau$. Then for each $\eta < \kappa$ we can find a subfamily $\mathcal{D}_\eta$ of 
$\mathcal{U}_\eta$ such that $|\mathcal{D}_\eta|\le \tau$ and 
$X = \bigcup\{\overline{D}:D\in \mathcal{D}_\eta\}$. 

Let $x\in X$. For each $\eta<\kappa$ we fix $D_{x,\eta}\in \mathcal{D}_\eta$ such that 
$x\in \overline{D}_{x,\eta}$. Now, let $y\in X\setminus\{x\}$. Then there is $\eta<\kappa$ and an 
open set $U_\eta(x,y) \in \sU_\eta$  containing $x$ such that 
$y\notin \overline{\st(U_\eta (x,y),\sU_\eta)}$. Since $D_{x,\eta}\in \mathcal{D}_\eta\subset \sU_\eta$, 
we have $D_{x,\eta}\subset \st(U_\eta(x,y),\sU_\eta)$. 
Hence $\overline{D}_{x,\eta}\subset \overline{\st(U_\eta (x,y),\sU_\eta)}$ and therefore 
$y\notin \overline{D}_{x,\eta}$. This shows that $\{x\}=\bigcap_{\eta<\kappa}\overline{D}_{x,\eta}$.

Since each $D_{x,\eta}$ could be chosen out of $\tau$ many sets, there are $\tau^\kappa$ such 
possible intersections. Therefore we conclude that $|X|\le \tau^\kappa$. 
\end{proof}

\begin{corollary}\label{CGAL2}
If $X$ is a Urysohn space with a regular $G_\delta$-diagonal then $|X|\le aL(X)^\omega$.
\end{corollary}

\begin{corollary}
The cardinality of every almost Lindel\"of Urysohn space with a regular $G_\delta$-diagonal 
does not exceed $2^{\omega}$.
\end{corollary}

Notice that the results in Corollary \ref{CGAL2} and in Corollary \ref{CGGD} strengthen, for the class of 
spaces with regular $G_\delta$-diagonals, the following Bella and Cammaroto result: Let $X$ be a 
Urysohn space. Then $|X|\le 2^{\chi(X)\cdot aL(X)}$ \cite{BelCam88}.

\subsection*{Acknowledgements}
The author expresses his gratitude to V. Tkachuk, A. Bella and G. Gruenhage for useful e-mail 
correspondence. 

\bibliographystyle{amsplain}

\end{document}